\documentclass[12pt]{article}
\usepackage[cp1250]{inputenc}
\usepackage{enumerate, amssymb, amsmath, amsthm, mathrsfs}
\newcommand{\Cov}{\textrm{Cov}}
\newcommand{\bE}{\mathrm{E}}
\newcommand{\Covth}{\textrm{\em Cov}}

\newcommand{\R}{\mathbb{R}}
\newcommand{\bN}{\mathbb{N}}
\newcommand{\jed}{\mathbf{1}}
\newcommand{\ID}[4]{i\langle#1,#2\rangle-\frac{1}{2}%
  \langle#3#1,#1\rangle+\int_{\mathbb{R}^d}%
  \bigl(e^{i\langle#1,u\rangle}-1-i\langle#1,u\rangle\cdot%
  \jed_{\{\|u\|\leq#4\}}(u)\bigr)\nu(du)}

\theoremstyle{definition}
\newtheorem{defn}{Definition}[section]

\newtheorem{thm}[defn]{Theorem}
\newtheorem{prop}[defn]{Proposition}
\newtheorem{rem}[defn]{Remark}
\newtheorem{cor}[defn]{Corollary}

\begin{document}
\begin{center}
{\Large\textbf{Processes with block-associated increments}}\\[3mm]
\textbf{Adam Jakubowski, Joanna Kar\l owska-Pik}\\
Faculty of Mathematics and Computer Science\\
Nicolaus Copernicus University\\
Toruń, Poland
\end{center}

\begin{abstract}
This paper is motivated by relations between association and
independence of random variables. It is well-known that for real
random variables independence implies association in the sense of
Esary, Proschan and Walkup, while for random vectors this simple
relationship breaks. We modify the notion of association in such a
way that any vector-valued process with independent increments has
also associated increments in the new sense --- association
between blocks.

The new notion is quite natural and admits nice characterization
for some classes of processes. In particular, using the covariance
interpolation formula due to Houdr\'{e}, P\'{e}rez-Abreu and
Surgailis, we show that within the class of multidimensional
Gaussian processes block-association of increments is equivalent
to supermodularity (in time) of the covariance functions.

We define also corresponding versions of weak association,
positive association and negative association. It turns out that
the Central Limit Theorem for weakly associated random vectors due
to Burton, Dabrowski and Dehling remains valid, if the weak
association is relaxed to the weak association between blocks.

\end{abstract}

\section{Introduction}
Random variables $X_1, X_2, \ldots, X_n$ are  {\em associated} if
\begin{equation}\label{eq1}
\Cov (f(X_1,X_2,\ldots,X_n), g(X_1,X_2,\ldots,X_n)) \geq 0,
\end{equation}
for each pair of functions $f, g : \R^n \to \R^1$, which are
non-decreasing in each coordinate and for which the above
covariance exists. This definition, due to Esary, Proschan and
Walkup \cite{Esary},  seems to be the most appropriate description
of positive dependence phenomena encountered in various areas,
e.g. reliability theory \cite{BP}, \cite{MuellerStoyan},
statistical physics \cite{Newman1}, \cite{Newman2},
\cite{Liggett}, multivariate extremes \cite{Resnick} or random
sets \cite{KPS}, to mention but a~few. We refer to the recent
monograph \cite{BS} for properties of association, an extensive
list of references and more abstract formalism of associated
random elements.

Our paper is motivated by relations between association and
independence of random variables. It is well-known \cite[Theorem
1.8]{BS}, that any family of independent random variables is
associated. In particular, any stochastic process
$X=\{X_t\}_{t\geq0}$ {\em with independent increments} has also
{\em associated increments} in the sense of Glasserman
\cite{Glasserman}. The last statement means that for any choice of
sampling points $0 < t_1<t_2<\ldots<t_n$ the differences
$$\Delta_1X=X_{t_1} - X_0,\quad
\Delta_2X=X_{t_2}-X_{t_1},\quad\ldots,\quad
\Delta_nX=X_{t_n}-X_{t_{n-1}}$$ are associated random variables.

This simple and natural relationship breaks when we pass to
processes with values in $\R^d$, $d \geq 2$. Consider, for
example, a real process $\{Z_t\}_{t\geq 0}$ with independent
increments and non-degenerate marginal laws and set
\[ Y_t = \left[\begin{array}{c}
Z_t \\
 - Z_t
 \end{array}\right].
 \]
Then $\{Y_t\}_{t\geq 0}$ retains independence of increments, but
clearly the components of each increment $\Delta_kY$ (e.g. $X_1 =
Z_{t_1} - Z_0$, $X_2 = -(Z_{t_1} - Z_0)$) do not satisfy
(\ref{eq1}), hence the random vectors  $\Delta_1Y,\Delta_2Y,
\ldots, \Delta_nY$ cannot be associated.

We aim at modifying the notion of association in such a way that
\begin{itemize}
\item for random variables ($d=1$) the new notion is equivalent to association;
\item any vector-valued process with independent increments has also
associated increments in the new sense.
\end{itemize}
This is done in Section 2, where we introduce {\em association
between blocks} of random variables. The idea consists in
requiring association between {\em real} non-decreasing (in each
coordinate) functions of blocks. It turns out that the modified
notion of association can be easily characterized within classes
of random vectors with multivariate normal or infinitely divisible
distributions (like the usual association). Similarly, when
applied to increments of stochastic processes, the new notion
admits nice characterizations within particular classes of
processes. For example, for multidimensional Gaussian processes
the block-association of increments is equivalent to
$L$-superadditivity (or supermodularity) of all covariance
functions (see Theorem \ref{thgaus}, Section 3). This example
shows that association between blocks deals with core properties
of multidimensional stochastic processes.

In a similar spirit, in Section 4 we weaken the notion of {\em
weak association} introduced by Burton, Dabrowski and Dehling
\cite{BDD}, {\em positive association} (as defined in Bulinski and
Shashkin \cite{BS}) and {\em negative association} (due to
Joag-Dev and Proschan \cite{JDP}). It is interesting that obtained
this way ``weak association between blocks" and ``positive
association between blocks" coincide while their prototypes
differ.

The weak association of random vectors  is formally stronger than
the weak association between blocks built upon coordinates of
vectors. We do not know any example showing that the equality of
both classes actually does not hold. On the other hand an
inspection of methods based on factorization of increasing
functions and used in the proof of Theorems \ref{thgau} and
\ref{thid}  suggests that verifying whether a sequence of random
vectors is ``weakly associated between blocks" may be essentially
easier then the corresponding procedure for ``weak association".
Therefore in Section~5 we restate a complete multidimensional
generalization of Newman's Central Limit Theorem \cite{Newman1}
and Newman-Wright's Invariance Principle \cite{NewmanWright} for
sums of  stationary associated random variables, originally proved
by Burton, Dabrowski and Dehling \cite{BDD} for weakly associated
random vectors. The point is that this result is valid under weak
association between blocks, without any change in its proof.

\section{Association between blocks}
In what follows when referring to vectors we mean {\em column}
vectors.

Let us consider a family $X=\{X_i, i\in I\}$ of real-valued random
variables indexed by a finite set $I$. Suppose that
$I=\bigcup_{k=1}^n I_k$, where the sets $I_k$ are non-empty and
pairwise disjoint. The sets $I_1,\ldots, I_n$ form {\em the
blocks' basis} $\mathcal{J}$. Equip each set $I_k$ with some
arbitrary (but fixed) linear order. Write $X(I_k)$ for vector with
components $\{X_i, i\in I_k\}$. Let $|I|$ denote the cardinality
of $I$.

We are ready to formulate our basic definition.

\begin{defn}
A family $X=\{X_i, i\in I\}$ is called \emph{associated between
blocks} if for all non-decreasing functions
$f_k:\R^{|I_k|}\rightarrow\R$, $k=1,2,\ldots,n$, the random vector
$$(f_1(X(I_1)),f_2(X(I_2)),\ldots,f_n(X(I_n)))$$
is associated, i.e. for all non-decreasing functions
$g,h:\R^n\rightarrow\R$
\begin{equation}\label{eq3}
\Cov (g(f_1(X(I_1)),\ldots,f_n(X(I_n))),
h(f_1(X(I_1)),\ldots,f_n(X(I_n))))\geq0
\end{equation}
if the above covariance is well defined. \label{def_blocks}
\end{defn}

The very definition and basic properties of association imply the
following facts.

\begin{prop} Let $X=\{X_i, i\in I\}$ be an
associated family of random variables. Then for arbitrary
partition $I=\bigcup_{k=1}^n I_k$ we have association of $X$
between blocks based on $I_1,\ldots,I_n$.
\end{prop}
\begin{prop}
If vectors $X(I_k)$, $k=1,2,\ldots,n$, are independent then $X$ is
associated between blocks.
\end{prop}

\begin{prop}
For a fixed blocks' basis $\mathcal{J} = \{I_1, I_2, \ldots,
I_n\}$, the family ${\mathcal{P}}^+_{\mathcal{J}}$ of laws of
random vectors which are associated between blocks based on $I_1,
I_2, \ldots, I_n$ is closed with respect to the topology of weak
convergence.
\end{prop}

Let $X = \{X_i\}_{i\in I}$ be an $|I|$-dimensional Gaussian random
vector. It is well known \cite{Pitt} --- but by no means trivial
--- that the non-negativity of all entries of the covariance
matrix $\Sigma$ of $X$ is necessary and sufficient for association
of $X$. We have a very similar situation for the association
between blocks.

\begin{thm}\label{thgau}
A Gaussian random vector $X = \{X_i\}_{i\in I}$ is associated
between blocks built on $I_1, I_2, \ldots, I_n$ if and only if
$\sigma_{kl} = \Covth(X_k, X_l) \geq 0$ for all $k, l$ which are
not in the same block.
\end{thm}

While the necessity part in the above theorem is obvious, the
sufficiency does not seem to be easy unless advanced tools are
used. We propose to exploit the covariance interpolation formula
and the technique developed by Houdr\'{e}, P\'{e}rez-Abreu and
Surgailis (\cite{HPS}, Section 2), restated below in Proposition
\ref{FormHPS}. Since the covariance formula is valid for general
infinitely divisible distributions, Theorem \ref{thgau} is a
direct consequence of Theorem \ref{thid}, which will be given
after a necessary notation is introduced.

Let $X=\{X_i\}_{i\in I}$ be an $|I|$-dimensional infinitely
divisible random vector with the L\'evy-Khinchin triplet
$(a,\Sigma,\nu)$ (we write then $X\sim\mathcal{ID}(a,\Sigma,\nu)$)
and the characteristic function
$\varphi(t)=\varphi(t;a,\Sigma,\nu)$ given by
\begin{equation}\label{ajeq1}
\ln \varphi(t) =\ID{t}{a}{\Sigma}{1}.
\end{equation}
Recall that $a \in \R^{|I|}$ is a vector,
$\Sigma=(\sigma_{kl})_{k,l\in I} \in \R^{|I|} \otimes \R^{|I|}$ is
the covariance matrix of the Gaussian component of $X$ and  $\nu$
stands for the L\'evy measure (for definitions related to infinite
divisibility we refer to \cite[Section 8]{Sato}). We shall
associate with $\nu$ its two-dimensional characteristics
$\nu_{kl}$. If\break
$\pi_{kl}:\R^{|I|}\rightarrow\R^2$ are
standard projections on $\R^2$, i.e.
$$\pi_{kl}(x_1,x_2,\ldots,x_{|I|})=(x_k,x_l),\quad 1\leq k<l\leq|I|,$$
we define $\nu_{kl}$ on $\R^2$ by the formula
\begin{equation}\label{dwawymiary}
 \nu_{kl}(A)  = (\nu \circ \pi_{kl}^{-1})\big(A\cap (\R^2 \setminus \{0\})\big).
\end{equation}

Notice that $\nu_{kl}$ is a L\'evy measure on $\R^2$, but it does
not have to be the two-dimensional projection of $\nu$.

A combination of results by Pitt \cite{Pitt} and Resnick
\cite{Resnick} states that non-negativity of all entries of
$\Sigma$ together with  the concentration of the L\'evy measure
$\nu$ on $(\R_+)^{|I|}\cup(\R_-)^{|I|}$ are enough for association
of $X$.  Theorem \ref{thid} establishes analogous conditions for
association between blocks of an infinitely divisible random
vector.

\begin{thm}\label{thid} Let $X \sim \mathcal{ID}(a,\Sigma,\nu)$.
If for all $k,l\in I$, which {\em are not in the same block},
\begin{description}
\item{(i)} $\sigma_{kl}$ are non-negative,
\item{(ii)} the measures $\nu_{kl}$ are concentrated on $(\R_-)^2 \cup (\R_+)^2$,
\end{description}
then $X$ is associated between blocks.
\end{thm}

Let $X\sim\mathcal{ID}(a,\Sigma,\nu)$ and let $\varphi$ be given
by
 (\ref{ajeq1}).
Define
$$\varphi_0(r,s)=\varphi(r)\varphi(s),\quad
\varphi_1(r,s)=\varphi(r+s),\quad r,s\in\R^{|I|}.$$ For each
$\alpha \in [0,1]$, let $(Y^{\alpha},Z^{\alpha})$ be an infinitely
divisible
 random vector of dimension $2|I|$ with distribution given by the
characteristic function
$$\varphi_{\alpha}(r,s)=\varphi_0^{1-\alpha}(r,s)\varphi_1^{\alpha}(r,s).$$
Then for each $\alpha \in [0,1]$ we have  $Y^{\alpha}\sim
Z^{\alpha} \sim X$ and the vector $(Y^{\alpha},Z^{\alpha})$
``interpolates" between independent copies $Y^0$, $Z^0$ of the
vector $X$ and  the totally dependent copies $Y^1=Z^1$ of $X$. We
are ready to restate the covariance formula due to Houdr\'e,
Perez-Abreu and Surgailis \cite{HPS}.

\begin{prop}\label{FormHPS}
For any functions $\psi_1,\psi_2\in\mathcal{C}^1_b(\R^{|I|})$
(continuously differentiable with bounded derivatives)
\begin{gather*}
\Covth(\psi_1(X),\psi_2(X))=\\
=\int_0^1\bE\left(\langle\Sigma\nabla\psi_1(Y^{\alpha}),
\nabla\psi_2(Z^{\alpha})\rangle+\int_{\R^{|I|}}\Delta_u
\psi_1(Y^{\alpha})\Delta_u\psi_2(Z^{\alpha})\;\nu(du)\right)\;d\alpha,
\end{gather*}
where $\nabla$ is the gradient operator and
$\Delta_u\psi(x)=\psi(x+u)-\psi(x)$.
\end{prop}

Now we can turn to the proof of Theorem \ref{thid}, keeping in
mind that it is enough to study (\ref{eq1}) only for functions
from ${C}^1_b(\R^{|I|})$ (see e.g. \cite[Theorem 1.5]{BS}).

\begin{proof} Choose non-decreasing and $C^1_b$ functions
$f_i:\R^{|I_i|}\rightarrow\R$, $k=1,2,\ldots,n$, and denote by $F$
the mapping from $\R^{|I|}$ into $\R^n$ given by
$$F(x_k, k\in I)=(f_1(x_k, k\in I_1),\ldots,f_n(x_k, k\in I_n)).$$
We will identify the functions $f_i$ with their corresponding
extensions\break
$\widetilde{f}_i(x) = f_i(\pi_{I_i}(x))$.

Let $g,h : \R^n \to \R$ be non-decreasing and $C^1_b$. Our goal is
to establish the sign of the covariance
\begin{equation}\label{covfor}
\begin{split}
\Cov(g(F(X)),&\;h(F(X)))=\\
&=\int_0^1\bE\big(\langle\Sigma\nabla(g\circ F) (
Y^{\alpha}),\nabla(h\circ F)
( Z^{\alpha})\rangle+\\
&\quad\quad +\int_{\R^{|I|}}\Delta_{u} (g\circ F) (
Y^{\alpha})\Delta_{u}(h\circ F) (Z^{\alpha})\;\nu(d
u)\big)\;d\alpha.
\end{split}\end{equation}
Applying the chain rule we get that $\nabla(g\circ F)(y)$ is the
product of the transposed matrix of partial derivatives of $F$ and
the vector $(\nabla g)(F(y))$. The first from these factors is the
matrix with $n$ columns and $|I|$ rows, with non-zero elements
only for $k\in I_i$ ($i$ is the kolumn and $k$ is the row number).
So
\[
\big(\nabla(g\circ F)(y)\big)_k=\begin{cases} \displaystyle
\frac{\partial f_i}{\partial x_k}(y)\frac{\partial g}{\partial
v_i}(F(y))& \text{if  $k\in
I_i,\; i=1,2,\ldots,n,$}\\
0 &\text{otherwise.}
\end{cases}
\]

Hence the scalar product in the covariance formula has the
following form.
\begin{align}
&\langle\Sigma\nabla(g\circ F) (y),\nabla(h\circ F) (z)\rangle= \nonumber \\
&\quad=\sum_{i=1}^n\sum_{j=1}^n\sum_{k\in I_i} \sum_{l\in I_j
}\sigma_{kl}\frac{\partial f_i}{\partial x_k}(y)\frac{\partial
g}{\partial v_i}(F(y))\frac{\partial f_j}{\partial x_l}(
z)\frac{\partial h}{\partial v_j}(F(z))= \nonumber\\
&\quad=\sum_{i=1}^n\frac{\partial g}{\partial v_i}(F(y))
\frac{\partial h}{\partial v_i}(F(z))\left(\sum_{k\in I_i}
\sum_{l\in I_i}\sigma_{kl}\frac{\partial f_i}{\partial x_k}(
y)\frac{\partial f_i}{\partial
x_l}(z)\right)+\label{suma1}\\
&\quad\quad+\sum_{i=1}^n\sum_{\substack{j=1\\j\neq i}}^n\sum_{k\in
I_i} \sum_{l\in I_j }\sigma_{kl}\frac{\partial f_i}{\partial
x_k}(y) \frac{\partial g}{\partial v_i}(F(y))\frac{\partial
f_j}{\partial x_l}(z)\frac{\partial h}{\partial v_j}(F(
z)).\label{suma2}
\end{align}
The expression in line (\ref{suma1}) is non-negative because the
partial derivatives are non-negative  and
$$\sum_{k\in I_i} \sum_{l\in I_i}
\sigma_{kl}\frac{\partial f_i}{\partial x_k}( y) \frac{\partial
f_i}{\partial x_l}(z)\geq0$$ due to the fact that $\sigma_{kl}$
for $k,l\in I_i$ are entries of the covariance matrix of the
vector $X(I_i)$. The expression in line (\ref{suma2}) is
non-negative for all partial derivatives are non-negative and
$\sigma_{kl} \geq 0$ if  $k,l$  are not in the same block.

It remains to check that the second summand in (\ref{covfor}) is
non-negative. Let us consider the following sets.
\begin{eqnarray*}
&&A_+ = \{u\,:\, F(y+u) \geq F(y)\}\cap \{u\,:\,F(z+u) \geq F(z)\}\\
&&A_- = \{u\,:\, F(y+u) \leq F(y)\}\cap \{u\,:\,F(z+u) \leq
F(z)\}.
\end{eqnarray*}
It is easy to see that on the set $A = A_+ \cup A_-$
\[\Delta_{u}(g\circ F)(y)\Delta_{u}(h\circ F)(z) = \big(g(F(y+u)) - g(F(y))\big)
\big(h(F(z+u)) - h(F(z))\big) \geq 0,\] for both factors are at
the same time either non-negative or non-positive. It follows that
it is enough to prove that
\begin{equation}\label{glowny}
\nu (A^c) = \nu (A_+^c \cap A_-^c) = 0,
\end{equation}
where $B^c$ is the complement of $B$. We have
\[ A_+ = \bigcap_{i=1}^n \{u\,:\, f_i(y+u)\geq f_i(y),\, f_i(z+u)\geq f_i(z)\},\]
hence
\[ A_+^c = \bigcup_{i=1}^n \{u\,:\, f_i(y+u) < f_i(y)\}\cup \{u\,:\, f_i(z+u) < f_i(z)\}\]
and similarly
\[ A_-^c = \bigcup_{j=1}^n \{u\,:\, f_j(y+u) > f_j(y)\}\cup \{u\,:\, f_j(z+u) > f_j(z)\}.\]

So $A^c = \bigcup_{1 \leq i \neq j \leq n}^n B_{ij}$, where
\begin{eqnarray*}
B_{ij} &=& \{u\,:\, f_i(y+u) < f_i(y), f_j(y+u) > f_j(y)\} \\
& &\cup\  \{u\,:\, f_i(y+u) < f_i(y), f_j(z+u) > f_j(z)\} \\
& &\cup\  \{u\,:\, f_i(z+u) < f_i(z), f_j(y+u) > f_j(y)\} \\
& &\cup\  \{u\,:\, f_i(z+u) < f_i(z), f_j(z+u) > f_j(z)\}.
\end{eqnarray*}

Since $f_i$'s are non-decreasing,  $f_i(x+u) < f_i(x)$ implies
that there exists $k\in I_i$ such that $u_k < 0$. (If $u$ were in
$(\R_+)^{I_i}$ we would have $f_i (x + u) \geq f_i(x)$).
Similarly, $f_i(x+u) > f_i(x)$ implies that there exists $l\in
I_i$ such that $u_l > 0$. Thus we obtain that
\[
B_{ij} \subset \bigcup_{k\in I_i} \bigcup_{l\in I_j}  \{u\,:\, u_k
<0, u_l > 0\}.
\]
But $i\neq j$ and so $k$ and $l$ in the above union of sets {\em
are not in the same block}. It follows that
\[ \nu (\{u\,:\, u_k < 0, u_l > 0\}) = \nu_{kl}\big((-\infty,0)\times(0,+\infty)\big) = 0.\]
Hence $\nu(B_{ij}) = 0$ and $\nu(A^c) = 0$.
\end{proof}

For future purposes we need a convenient reformulation of the
condition imposed in Theorem \ref{thid} on  the two-dimensional
L\'evy measures $\nu_{kl}$.

\begin{prop}\label{rownowazny}
 Let $\nu$ be a measure on $\R^{|I|}$ and let measures $\nu_{kl}$ on $\R^2$ be defined by
(\ref{dwawymiary}). Then the following statements are equivalent:
\begin{description}
\item {(i)} For all $k,l\in I$, which {\em are not in
the same block},  the measures $\nu_{kl}$ are concentrated on
$(\R_-)^2 \cup (\R_+)^2$,
 \item{(ii)} The measure $\nu$ is concentrated on the set
\end{description}
\begin{equation}\label{nosnik}
S = (\R_+)^{|I|}\cup(\R_-)^{|I|}\cup U,
\end{equation}
where
\[ U = \bigcup_{m=1}^{n}\left(\{0\}^{\sum_{i=1}^{m-1}|I_i|}\times\R^{|I_m|}
\times\{0\}^{\sum_{j=m+1}^n|I_j|}\right).\]
\end{prop}

\begin{proof}
It is clear that if $\nu$ concentrates on $S$ given in
(\ref{nosnik}), then $\nu_{kl}$ satisfy (i). Thus we have to prove
the implication (i) $\Rightarrow$ (ii) only. For notational
convenience, let us write $k\sim l$ if $k$ and  $l$ {\em are} in
the same block and $k\not\sim l$ otherwise. Let us also denote
\[D_k^+ = \{ x \in \R^{|I|}\,:\, x_k > 0\},\quad D_k^- = \{ x \in
\R^{|I|}\,:\, x_k < 0\}\] and
\[D = \bigcup_{(k,l)\,:\, k\not\sim l} D_k^+ \cap D_l^-. \]
Then (i) implies $\nu (D_k^+\cap D_l^-) = 0$ for all pairs $(k,l)$
such that $k\not\sim l$ and so
\begin{equation}\label{dezero}
\nu (D) = 0.
\end{equation}
Now (ii) follows from  (\ref{nosnik}), (\ref{dezero}) and the
observation that
\[ \R^{|I|} = (\R_+)^{|I|}\cup(\R_-)^{|I|}\cup U \cup D = S \cup D.\]
\end{proof}

The example given by Samorodnitsky \cite{Samorod} shows that there
exists an  associated (so associated between blocks of the length
1, too) random vector with 2-dimensional infinitely divisible
distribution and with L\'evy measure assigning a positive mass out
of the set $(\R_+)^2\cup(\R_-)^2$. So in Theorem \ref{thid} the
condition related to concentration of  measures $\nu_{kl}$  is not
necessary for association between blocks of the multidimensional
vector with infinitely divisible distribution.

On the other hand there exists a natural framework proposed by
Samorodnitsky \emph{ibid.} in which  the concentration of  the
L\'evy measure on  $(\R_+)^{|I|}\cup~(\R_-)^{|I|}$ {\em is}
necessary. The theorem below can be proved in much the same way as
Theorem 3.1 \emph{ibid.} or Proposition 3 in \cite{HPS}.

\begin{thm}
Let $X\sim\mathcal{ID}(a,\Sigma,\nu)$.   Let $\{X_t,\;t\geq0\}$ be
a L\'evy process with $X_1=_d X$. Then the following are
equivalent.
\begin{description}
\item{(i)} For every $t>0$ and any choice
of non-decreasing functions $f_1:\R^{|I_1|}\rightarrow~\R,$
$\ldots, f_n:\R^{|I_n|}\rightarrow\R$, the vector
$$(f_1((X_t)_{I_1}),\ldots,f_n((X_t)_{I_n}))$$
is associated.
\item{(ii)} For all indices $k$, $l$ which {\em are not in
the same block}, the entries $\sigma_{kl}$ of the matrix $\Sigma$
are non-negative and the L\'evy measures $\nu_{kl}$ concentrate on
the set $(\R_+)^{2}\cup(\R_-)^{2}$.
\end{description}
\end{thm}

\section{Block-association of increments of stochastic processes}

Let $\{X_t=(X_t^1,X_t^2,\ldots,X_t^d),t\in\R\}$ be a
$d$-dimensional stochastic process and let $0<
t_1<t_2<\ldots<t_n$. We can consider an $nd$-dimensional random
vector formed by the increments
$$X_{t_1}-X_0,\quad X_{t_2}-X_{t_1},\quad\ldots,\quad
X_{t_n}-X_{t_{n-1}}.$$ Such vector has naturally distinguished
blocks of the length $d$. The first is formed by the components of
$X_{t_1}-X_0$, the second by the components of $X_{t_2}-X_{t_1}$
and so on. Hence, according to Definition \ref{def_blocks}, we
have

\begin{defn}
A $d$-dimensional stochastic process $\{X_t,t\in\R\}$ \emph{has
block-associated increments} if for every $n\in\bN$ and any choice
of $0< t_1<t_2<\ldots<t_n$ the increments
$$X_{t_1}-X_0,\quad X_{t_2}-X_{t_1},\quad\ldots,\quad
X_{t_n}-X_{t_{n-1}}$$ form the vector associated between blocks.
\end{defn}

With such a definition we have the expected result.

\begin{thm}
Every process with independent increments has block-associated
increments.
\end{thm}

Next we shall discuss Gaussian processes.

\begin{thm}\label{thgaus}
Let $\{X_t,t\geq0\}$ be a $d$-dimensional Gaussian process with
the covariance functions $K^{k,l}(s,t)=\Covth(X_s^k,X_t^l)$,
$k,l=1,\ldots, d$. The process $\{X_t,t\geq0\}$ has
block-associated increments if and only if its covariance
functions are L-superadditive on $\{(s,t);s\leq t\}$, i.e.
$$K^{k,l}(s_1,t_1)-K^{k,l}(s_2,t_1)-K^{k,l}(s_1,t_2)+K^{k,l}(s_2,t_2)\geq0$$
for all $0 \leq s_1 \leq s_2\leq t_1\leq t_2$. \label{proc_gauss}
\end{thm}

\begin{proof}
Let us consider the $nd$-dimensional vector
$$(X_{t_1}^1-X_0^1,\ldots,X_{t_1}^d-X_0^d,\ldots,X_{t_n}^1-X_{t_{n-1}}^1,\ldots,X_{t_n}^d-X_{t_{n-1}}^d),$$
where $0<t_1<t_2<\ldots<t_n$. As we know from Theorem \ref{thgau},
the process $\{X_t, t\geq0\}$ has block-associated increments if
and only if for all $k,l=1,\ldots,d$ and  $1\leq i < j \leq  n$,
$i\neq j$ the covariances
$$\sigma_{ij}^{k,l}=\Cov(X_{t_i}^k-X_{t_{i-1}}^k,X_{t_j}^l-X_{t_{j-1}}^l)$$
are non-negative. But
\[\begin{split}
0\leq\sigma_{ij}^{k,l}&=\Cov(X_{t_i}^k-X_{t_{i-1}}^k,X_{t_j}^l-X_{t_{j-1}}^l)\\
&=K^{k,l}(t_i,t_j)-K^{k,l}(t_i,t_{j-1})-K^{k,l}(t_{i-1},t_j)+K^{k,l}(t_{i-1},t_{j-1}).
\end{split}\]
\end{proof}

\begin{rem}
The notion of L-superadditivity is well known, see for example
Marshall, Olkin \cite[Ch. 6, Sect. D]{MO}.
\end{rem}

\begin{cor}
If the covariance functions $K^{k,l}$ ($k,l=1,\ldots,d$) of
the\break $d$-dimensional Gaussian process $\{X_t,t\geq0\}$ are
continuously twice differentiable for $s\neq t$, then $\{X_t ,
t\geq0\}$ has block-associated increments if and only if
$$\frac{\partial^2 }{\partial s\,\partial t}K^{k,l}(s,t)\geq0 \textrm{ for } s\neq t \textrm{ and } k,l=1,2.\ldots,d.$$
\end{cor}

\begin{proof}
The L-superadditivity of the covariance functions is, under the
corollary's assumptions, equivalent to the non-negativity of the
mixed second derivatives. Indeed,
\[\begin{split}
&K^{k,l}(t_i,t_j)-K^{k,l}(t_i,t_{j-1})-K^{k,l}(t_{i-1},t_j)+K^{k,l}(t_{i-1},t_{j-1})=\\
&=\int_{t_{i-1}}^{t_i}\left(\frac{\partial K^{k,l}}{\partial
u}(u,t_j)-\frac{\partial K^{k,l}}{\partial
u}(u,t_{j-1})\right)\;du\\
&= \int_{t_{i-1}}^{t_i}\int_{t_{j-1}}^{t_j}\frac{\partial^2
K^{k,l}}{\partial v\,\partial u}(u,v) \;dv\,du
\end{split}\]
and $(t_{i-1},t_i)$, $(t_{j-1},t_j)$ are arbitrary disjoint
intervals in $(0,+\infty)$.
\end{proof}

Similarly as Theorem \ref{thgau} produced Theorem \ref{thgaus},
one could also use Theorem \ref{thid} for writing a corresponding
result for infinitely divisible processes (processes with
infinitely divisible finitely dimensional distributions --- see
e.g. Maruyama \cite{Mar70} or Rajput and Rosinski \cite{RR89}). We
shall do that in a~special case and using Proposition
\ref{rownowazny}.

\begin{thm}
Let $\{X_t,t\geq0\}$ be a $d$-dimensional infinitely divisible
stochastic process. Let us suppose that for every choice of $0 =
t_0 < t_1<t_2<\ldots<t_n$ the distribution of $(X_0, X_{t_1},
X_{t_2}, \ldots, X_{t_n})$ doesn't have the Gaussian component and
the support of its  L\'evy measure $\nu_{0,t_1,\ldots,t_n}$ is
contained in the set
\[\begin{split}
\{(x_0,x_1,\ldots, x_n)\,:\, & x_0\leq x_1\leq x_2\leq\ldots\leq
x_n
\textrm{ or } x_0\geq x_1\geq x_2\geq\ldots\geq x_n\\
&\textrm{ or } x_1=x_2=\ldots=x_n\ \textrm{ or for some }m=2,3,\ldots, n \\
& x_0 = x_1 = \ldots = x_{m-1},\, x_m = x_{m+1}=\ldots = x_n\},
\end{split}\]
where $x_0, x_1,x_2,\ldots, x_n$ are $d$-dimensional vectors and
$\leq$ and $\geq$ are coordinate-wise inequalities.

Then $\{X_t,t\geq0\}$ has block-associated increments.
\end{thm}

\begin{proof}
Let $U:\R^{(n+1)d}\rightarrow\R^{nd}$ be given by the formula
$$U(x)=U(x_0, x_1, \ldots, x_n)=(x_1-x_0,x_2-x_1,\ldots, x_n-x_{n-1}).$$
It is well-known that if $(X_0,X_{t_1},X_{t_2},\ldots,X_{t_n})$
has an infinitely divisible distribution with a L\'evy measure
$\nu_{0,t_1,\ldots,t_n}$ then the vector of increments\break
$(X_{t_1}-X_0, X_{t_2}-X_{t_1},\ldots,X_{t_n}-X_{t_{n-1}})$ has
also an infinitely divisible distribution with the L\'evy measure
$\nu_{0,t_1,\ldots,t_n}\circ U^{-1}$ (up to an atom at $0$, see
e.g. Sato \cite[Proposition 11.10]{Sato}). For the
block-association of increments it is enough that the L\'evy
measures $\nu_{0,t_1,\ldots,t_n}\circ U^{-1}$ concentrate on\break
$S =
(\R_+)^{nd}\cup(\R_-)^{nd}\cup\bigcup_{m=1}^{n}(\{0\}^{(m-1)d}\times\R^{d}\times\{0\}^{(n-m)d})$
(Proposition \ref{rownowazny}), so for $\nu_{0, t_1, \ldots, t_n}$
it is enough to concentrate on the union of sets
\begin{eqnarray*}
\lefteqn{\{x\,:\,x_1-x_0\geq0,x_2-x_1\geq0,\ldots,x_n-x_{n-1}\geq0\}}\qquad\\
&\cup& \{x\,:\,x_1-x_0\leq0,x_2-x_1\leq0,\ldots,x_n-x_{n-1}\leq0\}\\
&\cup& \{x\,:\, x_2-x_1 = 0, \ldots, x_n-x_{n-1} = 0\}\\
&\cup &\bigcup_{m=2}^n \{x\,:\,x_1-x_0=0, \ldots, x_{m-1}-x_{m-2}=0\}\\
&&\qquad\qquad  \cap\ \{x\, :\, x_{m+1}-x_{m}=0, \ldots, x_n -
x_{n-1}=0\}
\end{eqnarray*}
which equals to
\[\begin{split}
&\{x\,:\,x_0\leq x_1\leq x_2\leq\ldots\leq x_n
\textrm{ or } x_0\geq x_1\geq x_2\geq\ldots\geq x_n\\
&\quad\quad\textrm{ or }x_1=x_2=\ldots=x_n \textrm{ or for some $m=2,\ldots, n$}\\
&\quad\quad x_0=x_1=\ldots=x_{m-1}, x_{m}=x_{m+1}=\ldots=x_{n} \}.
\end{split}\]
\end{proof}

\begin{rem}\label{UwagaMaruyama}
It is clear that the finite dimensional properties of the L\'evy
measures
 $\nu_{t_0,t_1,\ldots,t_n}$ can be expressed in terms of their projective limit
$\nu$ (see \cite{Mar70}): $\nu$ must be concentrated on the union
of sets consisting of non-decreasing trajectories, non-increasing
trajectories and rather mysterious trajectories admitting only one
jump.
\end{rem}

\section{Some other notions of relaxed association}

The following notion was introduced by Burton et al. \cite{BDD}.
\begin{defn}\label{wa}
A sequence of $d$-dimensional random vectors $(X_1, X_2, \ldots,
X_m)$ is said to be {\em weakly associated} if whenever $\pi$ is a
permutation of $\{1,2,\ldots,m\}$, $1\leq k <m$   and $g : \R^{kd}
\to \R$, $h : \R^{(m-k)d} \to \R$ are coordinate-wise
non-decreasing, then
\[ \Cov\big(g(X_{\pi(1)},X_{\pi(2)}, \ldots, X_{\pi(k)}), h(X_{\pi(k+1)},X_{\pi(k+2)},\ldots, X_{\pi(m)}\big) \geq 0,\]
if the covariance exists. A family of random vectors is weakly
associated if its every finite subfamily is weakly associated.
\end{defn}
Burton et al. {\em ibid., Theorem 1},  provided an example of a
sequence
 of weakly associated random variables ($d=1$), which are not associated.
Let $Y_1, Y_2, \ldots $ be such a sequence. Fix $d > 1$ and define
a sequence of\break
$d$-dimensional random vectors  by
\[ X_k = (\underbrace{Y_k, Y_k, \ldots, Y_k}_{\text{$k$ times}}).\]
Then it is easy to see that $X_1, X_2, \ldots $ is weakly
associated but it is not associated between blocks built upon
coordinates. The following definition is in the spirit of Section
2.
\begin{defn}\label{wabb}
A family $X=\{X_i, i\in I\}$ is called \emph{weakly associated
between blocks} if for all non-decreasing functions
$f_k:\R^{|I_k|}\rightarrow\R$, $k=1,2,\ldots,n$, the random vector
$$(f_1(X(I_1)),f_2(X(I_2)),\ldots,f_n(X(I_n)))$$
consists of weakly associated random variables.
\end{defn}

The next definition can be found in Bulinski and Shashkin
\cite{BS}.
\begin{defn}
A family $\mathbf{X}=\{X_i, i\in I\}$ is called \emph{positively
associated}, if
$$\Cov \big(g(X(A_g)),h(X(A_h))\big)\geq0$$ for any disjoint sets
$A_g,A_h\subseteq I$ and all non-decreasing functions
$g:\R^{|A_g|}\rightarrow~\R$, $h:\R^{|A_h|}\rightarrow\R$.
\end{defn}

Clearly, for families of random variables ($d=1$) the notions of
weak association and positive association coincide. It is
interesting that due to this coincidence,
 the notions of weak association between blocks and
positive association between blocks are also the same. In fact, a
definition for the latter should look as follows.

\begin{defn}
A family $\mathbf{X}$ is called \emph{positively associated
between blocks}, if for all non-decreasing functions
$f_k:\R^{|I_k|}\rightarrow\R$, $k=1,2,\ldots,n$, the vector
$$(f_1(X({I_1})),f_2(X({I_2})),\ldots,f_n(X({I_n})))$$
is positively associated, i.e. for any disjoint finite sets
$A_g,A_h\subset\{1,2,\ldots,n\}$ and any non-decreasing functions
$g:\R^{|A_g|}\rightarrow\R$, $h:\R^{|A_h|}\rightarrow\R$
\begin{equation}\label{pabb}
\Cov \big(g(f_i(X({I_i})),i\in A_g), h(f_j(X({I_j})), j\in
A_h)\big)\geq 0,
\end{equation}
if the covariance exists.
\end{defn}
We see that both (\ref{pabb}) and Definition  \ref{wabb} state
that the random variables $f_1(X(I_1))$,
$f_2(X(I_2)),\ldots,f_n(X(I_n))$ are weakly associated, so there
is no need to define positive association between blocks.

\begin{rem}
It is easy to see that for jointly Gaussian random variables the
two types of relaxed association considered in the present paper
(association between blocks and weak association between blocks)
coincide and are equivalent to non-negativity of covariances of
random variables which {\em are not in the same block}.
\end{rem}

Next we shall give a formal statement of the original form and a
relaxed form of negative association due to Joag-Dev and Proschan
\cite{JDP}.

\begin{defn}
A family $\mathbf{X}=\{X_i, i\in I\}$ is called \emph{negatively
associated} if
$$\Cov \big(g(X(A_g)),h(X(A_h))\big)\leq0$$
 for any disjoint sets
$A_g,A_h\subseteq I$ and all non-decreasing functions
$g:\R^{|A_g|}\rightarrow~\R$, $h:\R^{|A_h|}\rightarrow\R$.
\end{defn}

\begin{defn}
A family $\mathbf{X}$ is called \emph{negatively associated
between blocks} if for all non-decreasing functions
$f_k:\R^{|I_k|}\rightarrow\R$, $k=1,2,\ldots,n$, the vector
$$(f_1(X_{I_1}),f_2(X_{I_2}),\ldots,f_n(X_{I_n}))$$
is negatively associated, i.e. for any disjoint finite sets
$A_g,A_h\subset\{1,2,\ldots,n\}$ and any non-decreasing functions
$g:\R^{|A_g|}\rightarrow\R$, $h:\R^{|A_h|}\rightarrow\R$
$$\Cov \big(g(f_i(X_{I_i}),i\in A_g), h(f_j(X_{I_j}), j\in A_h)\big)\leq0,$$
if the covariance exists.
\end{defn}

We conclude this section with definition of the corresponding
notions for increments of processes.

\begin{defn}
A $d$-dimensional stochastic process $\{X_t,t\geq0\}$ has
\emph{block-weakly-associated} (resp.
\emph{block-negatively-associated}) \emph{increments} if for every
$n\in\bN$ and any choice of $0<t_1<t_2<\ldots<t_n$ the increments
$$X_{t_1}-X_0,\quad  X_{t_2}- X_{t_1},\quad\ldots,\quad
 X_{t_n}- X_{t_{n-1}}$$ form the sequence of vectors which are weakly (resp. negatively) associated between
blocks formed by the $d$ components of each increment $X_{t_i} -
X_{t_{i-1}}$.
\end{defn}

\section{Limit theorems under weak association between blocks}

Let $X_1,X_2,\ldots$ be a sequence of $d$-dimensional random
vectors. After building blocks upon the coordinates of consecutive
vectors  we may compare the notions of weak association of random
vectors $\{X_k\}$ (Definition \ref{wa})
 and weak association between blocks (Definition \ref{wabb}). Formally the latter is weaker: in place of
non-decreasing  functions $g$ and $h$ ``directly" acting on
vectors:
\[g(X_{\pi(1)},X_{\pi(2)}, \ldots, X_{\pi(k)}),\quad  h(X_{\pi(k+1)},X_{\pi(k+2)},\ldots, X_{\pi(m)}), \]
the latter definition operates with factorizations
\[g\big(f_{\pi(1)}(X_{\pi(1)}),\ldots,f_{\pi(k)}(X_{\pi(k)})\big),\quad
h\big(f_{\pi(k+1)}(X_{\pi(k+1)}),\ldots,f_{\pi(m)}(X_{\pi(m)})\big).\]

As already mentioned in Introduction, we are not able to exhibit
any example of a sequence $\{X_k\}$, which is weakly associated
between blocks, but not weakly associated. On the other hand, the
computations performed in Section 2 and based on the covariance
interpolation formula suggest that it might be a serious advantage
to deal with factorized functions while checking whether the
sequence is weakly associated between blocks. This is one reason
for including the present section into the paper.

The other reason is that the complete generalization of Newman's
Central Limit Theorem \cite{Newman1} and Newman-Wright's
Invariance Principle \cite{NewmanWright} for sums of  stationary
associated random variables, originally proved by Burton,
Dabrowski and Dehling \cite{BDD} for weakly associated random
vectors, remains valid under weak association between blocks,
without any change in its proof. Here ``complete generalization"
means including as a particular case the Central Limit Theorem for
i.i.d. random vectors, with covariance matrices possibly
containing negative entries.

\begin{thm}\label{thm_graniczne_bloki}
Let $X_1,X_2,\ldots$ be a strictly stationary sequence of
$d$-dimensional random vectors, which are {\em weakly associated
between blocks}  and let $S_n=X_1+X_2 + \ldots+X_n$.

If $\bE X_1=0$, $\bE \|X_1\|^2<+\infty$ and
$\sum_{j=2}^{\infty}\bE X_1^kX_j^l<+\infty$ for all
$k,l=1,\ldots,d$ (where $X_j^k$ is the $k$-th component of the
vector $X_j$), then
$$\frac{S_n}{\sqrt{n}}\xrightarrow[n\rightarrow\infty]{\mathscr{D}}\mathcal{N}(0,\Sigma)$$
where $\Sigma=(\sigma_{kl})_{k,l=1\ldots,d}$ and $\sigma_{kl}=\bE
X_1^kX_1^l+2\sum_{j=2}^{\infty}\bE X_1^kX_j^l.$

Moreover, if
\[ Y_n(t)  = \frac{1}{\sqrt{n}} S_{[nt]},\ t \in \R^+,\]
(or $Y_n(t)$ is a polygonal interpolation between points $(k/n,
S_k/\sqrt{n})$), then
\[ Y_n \xrightarrow[n\rightarrow\infty]{\mathscr{D}} W_{\Sigma},\]
on the function space $C(\R^+:\R^d)$, where $W_{\Sigma}$ is a
Wiener process with covariance matrix $\Sigma$.
\end{thm}

\begin{proof}
In their proof, Burton, Dabrowski and Dehling \cite{BDD} use the
weak association of the following random variables:
$$f_j(X_j)=\langle a_j,X_j\rangle=\sum_{k=1}^d a_j^k X_j^k,$$
 where $a_j^1, a_j^2, \ldots, a_j^d \geq 0$ are suitably chosen (for tightness purposes, convergence of
finite dimensional distributions etc.). Our assumption on weak
association between blocks provides exactly the same information.
\end{proof}

\begin{rem}
It is likely that also other existing limit theorems for
associated random variables (see e.g.  \cite[Chapter 3]{BS}) can
be proved under relaxed assumptions like weak association between
blocks and in a similar way  as Theorem \ref{thm_graniczne_bloki}.
In particular, there is a work in progress towards results on
convergence to stable laws with infinite variance, paralleling
\cite{DJ}.
\end{rem}

\begin{center}
\textbf{Acknowledgement}
\end{center}

At some stage of preparation of the paper we had fruitful
discussions with Tomasz Schreiber. He passed away on December 1st,
2010.  We dedicate this paper to His memory.

\mbox{} \hfill {\small \begin{tabular}{l}
\textsc{Nicolaus Copernicus University}\\
\textsc{Faculty of Mathematics and Computer Science}\\
\textsc{ul. Chopina 12/18, 87-100 Toruń, Poland}\\
\textsc{e-mail:} adjakubo@mat.umk.pl, joanka@mat.umk.pl
\end{tabular}}

\end{document}